\documentclass[a4paper,11pt,english]{amsart}

\usepackage{babel,varioref}
\usepackage[ansinew]{inputenc}
\usepackage{amsmath}
\usepackage{amsfonts}
\usepackage{amscd}
\usepackage{enumerate}
\usepackage{graphics}
\usepackage{babel,varioref}
\usepackage{amscd}
\usepackage[pdftex]{graphicx}
\DeclareGraphicsExtensions{.pdf,.png,.jpg}

\input amssym.def
\input amssym.tex

\def\reels{\mathbb{R}}

\def\nac{\check{\nabla}}
\def\nah{\hat{\nabla}}
\def\cx{\check{X}}
\def\cy{\check{Y}}
\def\cz{\check{Z}}
\def\ct{\check{T}}
\def\ccr{\check{R}}
\def\cs{\check{S}}

\def\cg{\check{\gamma}}
\def\hv{\hat{V}}\def\hu{\hat{U}}
\def\hw{\hat{W}}
\def\ha{\hat{A}}
\def\hb{\hat{B}}
\def\hc{\hat{C}}

\def\hr{\hat{R}}
\def\hs{\hat{S}}

\def\hphi{\hat{\Phi}}
\def\hpsi{\hat{\Psi}}

\newtheorem{theo}{Theorem}
\newtheorem{prop}[theo]{Proposition}
\newtheorem{corr}[theo]{Corollary}
\newtheorem{lem}[theo]{Lemma}
\newtheorem{defi}{Definition}

\newtheorem*{exs}{Examples}

\newtheorem{rem}{Remark}
\newtheorem*{nota}{Notations}

\title[Can. torsion-free connections on the total spaces $TM$ and $T^*M$]{Canonical torsion-free connections on the total space of the tangent and the cotangent bundle}
\author[Lionel Bérard Bergery]{Lionel Bérard Bergery}
\author[Tom Krantz]{Tom Krantz}
\thanks{e-mail: Lionel.Berard-Bergery@iecn.u-nancy.fr, Tom.Krantz@iecn.u-nancy.fr\\
Address: Institut Élie Cartan Nancy. Université Henri Poincaré Nancy
1.\\ B.P. 239, F-54506 Vandoeuvre-lès-Nancy Cedex, France.}

\makeindex 

\begin{document}
\begin{abstract}
In this paper we define a class of torsion-free connections on the total space of the (co-)tangent bundle over a base-manifold with a connection and for which tangent spaces to the fibers are parallel. Each tangent space to a fiber is flat for these connections and the canonical projection from the (co-)tangent bundle to the base manifold is totally geodesic.
In particular cases the connection is metric with signature (n,n) or symplectic and admits a single parallel totally isotropic tangent n-plane.
\end{abstract}
\maketitle
\date{\today}
\section{Introduction}
The classification of holonomy is far from being complete by today. In the semi-riemannian case and neutral signature there is little known beside the results of the paper~\cite{LBBAI}. In this paper in particular the candidates for indecomposable torsion-free holonomy of signature $(2,2)$ are listed. Metrics are constructed for
the case where the holonomy admits two totally isotropic and complementary invariant spaces but there was no construction yet for metrics for the holonomy groups admitting a single totally isotropic invariant space. T. Leistner and A. Galaev constructed for all but one of these missing cases metrics.
The construction in this paper gives a general answer for all the missing cases in signature $(2,2)$. The construction works not only for signature $(n,n)$ giving holomies admitting a single totally isotropic invariant n-plane.
In the general context of torsion-free connections, the construction gives a lot of examples of manifolds with a torsion-free connection such that the holonomy admits a single invariant subspace with given dimension and fixed holonomy representation sub- and dual quotient representation. The construction can be generalized to other vector bundles. We recall the general results but omit the proofs of these which will be accessible in another paper concerning the general construction.

\section{Definitions}
\subsection{Fibered spaces}
\begin{nota}Let $B$ be a (finite-dimensional) manifold equipped with a torsion-free connection $\nac$
on the bundle $TB$. Let $\check{R}$ be its curvature tensor defined
by $\check{R}(\cx,\cy):= \nac_{[\cx,\cy]}-[\nac_{\cx},\nac_{\cy}]$
for $\cx$ and $\cy$ sections of $TB$.

Let $(M,\pi,B)$ be a (finite-dimensional) vector bundle over $B$ with projection $\pi: M \to B$
and fiber $F$. For $b\in B$ let $F_b$ be the fiber $\pi^{-1}(b)$.
Let $\nah$ be a connection on the vector bundle $M$. Let $\hr$ be
the curvature tensor of $\nah$.
\end{nota}

\begin{exs}
\begin{itemize}
\item The first example we will consider is the case $M=TB$ equipped with the connection $\nah=\nac$.
\item We will consider also the case $M=T^*B$ (equipped with the
dual connection $\nah=\nac^*$ of $\nac$).

Recall that for a section $\xi$ of $T^*B$ and sections $\cx$, $\cy$,
$\cz$ of $TB$ we have:

$$(\nac^*_{\cx} \xi)(\cy)=\cx\cdot\xi(\cy)-\xi(\nac_{\cx} \cy),$$

and $$(\hr(\cx,\cy)\xi)(\cz)=-\xi(\ccr(\cx,\cy)\cz).$$

\end{itemize}
\end{exs}

A section $X$ of $TM$ is said to be {\em $\pi$-related} to a
section $\cx$ of $TB$ if $\forall x\in M,
T_x\pi(X_x)=\cx_{\pi(x)}$.

For $X$ (resp. $Y$) sections of $TM$, $\pi$-related to $\cx$ (resp. $\cy$), $[X,Y]$ is $\pi$-related to
$[\cx,\cy]$.
For a proof of the latter statement see \cite{spivak} or \cite{besse} chapter 9.

Let $\mathcal{V}$ be the "vertical" distribution:
$\mathcal{V}_x=\{v\in T_xM \; | \; (T\pi)_x(v)=0\}$.

A section $V$ of $TM$ is said to be {\em vertical} if $\forall x\in
M$, $V_x\in \mathcal{V}_x$.

A section of $TM$ is clearly vertical if and only if it is $\pi$-related to
the null section of $TB$.

Note that from the preceding follows:
\begin{itemize}
\item $[V,W]$ is vertical if $V$ and $W$ are vertical.
\item $[X,V]$ is vertical if $V$ is vertical and the section $X$ of $TM$ $\pi$-related to some section $\cx$ of $TB$.
\end{itemize}

To the
connection $\nah$ corresponds a unique "horizontal" distribution
$\mathcal{H}$ with the following properties:

\begin{enumerate}[i)]
\item $\forall x\in M, T_xM=\mathcal{V}_x \oplus \mathcal{H}_x$

\item Let $\gamma:[0,1]\to M$ be a $C^\infty$ curve. We have: $\sigma=\pi \circ \gamma$ is a curve of $B$.
$\gamma$ can be seen as a section of the fiber bundle $(M,\pi,B)$ over the curve $\sigma$. The connection $\nah$ allows to derive sections over a curve and to define parallel transport $\tau_\sigma$ along the curve $\sigma$ by: $\tau_\sigma(t,\xi)$ (for $t\in [0,1]$ and $\xi\in F_{\sigma(0)}$) is equal to $X_t\in F_{\sigma(t)}$ where $X$ verifies $(\nabla_{\dot{\sigma}} X)_t=0$, $\forall t\in [0,1]$ and $X_0=\xi$.
The property $\mathcal{H}$ has to verify is: For every smooth curve $\gamma$ we have
$\forall t\in ]0,1[$, $\gamma_t=\tau_{\sigma}(t,\gamma_0)$ if and only if $\forall t\in
]0,1[$, $\dot \gamma_t \in \mathcal{H}_{\gamma_t}$.

If, as in our case $\nah$ is a linear connection, then $\mathcal{H}$ has to be "linear" in the following sense:
$\tau_\sigma(t,\alpha \xi+\beta \xi')= \alpha \tau_\sigma(t,\xi)+\beta \tau_\sigma(t,\xi')$ for all $t\in [0,1]$, $\alpha$ and $\beta$ reals, $\xi$ and $\xi'$ two vectors of $F_{\sigma(0)}$.
\end{enumerate}

A section $X$ of $TM$ is said to be {\em horizontal} if $\forall
x\in M, X_x\in \mathcal{H}_x$.

A section $X$ of $TM$ is said to be {\em basic} if it is horizontal
and $\pi$-related to a section of $TB$.

Note that $\mathcal{H}[X,Y]$ is basic if $X$ and $Y$ are basic.
We will see later that $\mathcal{V}[X,Y]$ is "tensorial" for $X$ and $Y$ basic {\em i.e.} it only depends on the value of $X$ and $Y$ on the corresponding point.

For any section $\cx$ of $TB$ there is a unique basic section $X$ of
$TM$ $\pi$-related to $\cx$.
For any section $Y$ of $TM$ $Y$ is $\pi$-related to $\cx$, if and only if there is a vertical section $V$ of $TM$ such that $Y=X+V$

\subsection{Special vector fields on $M$}

\subsubsection{Vertical vector fields constant along the fibers} For
$b\in B$ and $x\in F_b$, define the translation $t_x: F_b\to F_b$ by
$t_x(y)=x+y$ and note $O_b$ the origin of the vector space $F_b$.

We can identify $T_x F_b$ with $\mathcal{V}_x$ through the inclusion
$F_b\subset M$.

\begin{defi}
A vertical section $V$ of $TM$ is said to be {\em constant along the
fibers}(or simply {\em constant}) if for all $b\in B$ and $x\in F_b$,
$V(x)=t_x(V(O_b))$.
\end{defi}

A simple fact is the following:
\begin{prop}
For each section $\hv$ of the fibered space $(M,\pi,B)$ there is a unique vertical section $V$
of $TM$ which is constant along the fibers and such that
$\hv(b)=V(O_b)$ (with the identification $F_b \simeq T_{O_b} F_b
\simeq \mathcal{V}_{O_b}$). We say that $V$ is associated to $\hv$.
\end{prop}

\subsubsection{Linear vertical vector fields} We will note $End_B(M)$ the endomorphism fiber bundle of $M$ over $B$.
Let $\ha$ be a section
of $End_B(M)$ {\em i.e.} for all $b\in B$, $\ha_b\in End(F_b)$ and
$\ha_b$ depending $C^\infty$ of $b$.

For $b\in B$ and $x\in F_b$, we can identify $F_b$ with $T_x F_b$ by
the mapping $Id_x$ defined to be the composition of the canonical
isomorphism $F_b \simeq T_{O_b} F_b$ and the isomorphism $T_{O_b}
t_x: T_{O_b} F_b \to T_x F_b$.

\begin{defi} Let $\ha$ be a a section of $End_B(M)$.
A vertical section $A$ of $TM$ is said to be {\em linear along the
fibers associated to $\ha$}(or simply {\em linear}) if
$A_x=Id_x(\ha(x))$.
\end{defi}

\begin{lem}\label{linear_constant} A vertical section $A$ of $TM$ is linear if and only if
$[A,V]$ is constant for all $V$ constant, and for all $b\in B$,
$A(O_b)=0$.
\end{lem}
\begin{proof}
Can be checked in a local trivialization of the bundle $M$.
\end{proof}

\subsubsection{Affine vertical vector fields}
An {\em affine vertical vector field} is the sum of a constant vertical vector field with a linear vertical vector field.

\begin{nota} In the following
$X, Y, Z, ...$ will denote basic vector fields, and $\cx, \cy, \cz, ...$
the sections of $TB$ $\pi$-related to the preceding. Constant
vertical vector fields will be noted $U, V, W$ and $\hu, \hv, \hw$
their associated sections of $M$. Linear vertical vector fields will
be noted $A, B, C$ and $\ha, \hb, \hc$ the associated sections of
$End_B(M)$.
\end{nota}

\subsection{Basic properties}
\begin{prop}\label{basic_results}
For $V$, $W$ vertical constant along the fibers corresponding to the
sections of $M$, $\hv$, $\hw$, for $A$, $B$ vertical linear
associated to the sections of $End_B(M)$, $\ha$, $\hb$, for $X$, $Y$
basic $\pi$-related to $\cx$, $\cy$, we have:
\begin{enumerate}[(i)]
\item $[V,W]=0$;
\item $[A,V]$ is vertical and constant associated to
$-\ha(\hv)$;
\item $[A,B]$ is vertical and linear associated to $-[\ha,\hb]$;
\item $[X,V]$ is vertical and constant associated
to $\nah_{\cx}\hv$;
\item $[X,A]$ is vertical and linear associated
to $\nah_{\cx}\ha$;
\item The horizontal component $\mathcal{H}[X,Y]$ of $[X,Y]$ is
basic $\pi$-related to $[\cx,\cy]$;
\item The vertical component $\mathcal{V}[X,Y]$ of $[X,Y]$ is linear associated to $\hr(\cx,\cy)$.
\end{enumerate}
\end{prop}

\section{A family of connections on the manifold $M$}

\subsection{Definition}\label{connection_D}
We define a torsion-free connection $D$ on the bundle $TM$ by the
following equalities:

Fix a section $\hphi$ of $S^2(T^*B)\otimes End_B(M)$.

For $X,Y$ basic sections of $TM$ and $V,W$ constant vertical
sections of $TM$:

$$\begin{array}{l}
D_V W= 0\\
D_V X= 0\\
D_X V= [X,V]\\
D_X Y= \mathcal{H}D_X Y + \mathcal{V}D_X Y,\\
\end{array}$$
with $\mathcal{H}D_X Y$ basic $\pi$-related to $\nac_{\cx}\cy$ and
$\mathcal{V}D_X Y$ linear associated to
$\frac{1}{2}\hr(\cx,\cy)+\hphi(\cx,\cy)$.

\begin{exs}
\begin{itemize}
\item
$M$ is the vector bundle
$TB$ equipped with the connection $\nah$ equal to $\nac$ and $\hphi$ is chosen to be zero.

\item $M$ is the vector bundle $T^*B$ equipped with the dual connection $\nah$ of $\nac$.
Different choices are possible for $\hphi$. We will have a closer look at some possibility in the next paragraph.
\end{itemize}
\end{exs}

Define $\Gamma(\cx,\cy)(\cz):=\frac{1}{2}(\ccr(\cx,\cz)\cy+\ccr(\cy,\cz)\cx)$.
$\Gamma$ verifies the Bianchi-identity and is symmetric in the first two arguments.
Note that one can recover $\ccr$ from $\Gamma$: $\ccr(\cx,\cy)\cz=\frac{2}{3}(\Gamma(\cx,\cy)\cz-\Gamma(\cx,\cz)\cy)$.

\begin{rem} Note that more generally every $3$-tensor can canonically be decomposed in a totally antisymmetric, a totally symmetric, and a tensor verifying the Bianchi identity. Furthermore the vector space of $3$-tensors verifying the Bianchi identity is the direct sum of any two spaces from the set $\{\mathcal{BA}_{12},\mathcal{BA}_{13},\mathcal{BA}_{23},\mathcal{BS}_{12},\mathcal{BS}_{13},\mathcal{BS}_{23}\}$, where $\mathcal{BA}_{ij}$ (resp. $\mathcal{BS}_{ij}$) is the vector space of $3$-tensors verifying the Bianchi identity and antisymmetric (resp. symmetric) in the positions $i$ and $j$.
\end{rem}

\subsection{Structures on $T^*B$}
For $M=T^*B$ and $t\in\reels$, let $\Gamma^*(\cx,\cy)(\hv)(\cz):=-\hv(\Gamma(\cx,\cy)(\cz))$ and
$\hphi_{t} :=t\Gamma^*$.

\subsubsection{An example: a canonical pseudo-Riemannian structure on $T^*B$}
\label{cotangent} We can endow $M=T^*B$ with a
pseudo-Riemannian structure $\langle\cdot,\cdot\rangle$ in the
following way: For $V,W$ vertical sections of $TM$ associated to
$\hv$, $\hw$ and $X,Y$ basic sections of $TM$ $\pi$-related to
$\cx$, $\cy$ fix: $\langle V, W \rangle=0$,
 $\langle X, Y \rangle=0$, $\langle X, V \rangle_x= \hv(\cx)_{\pi(x)}$.

The Levi-Civita-connection $D$ corresponding to this
pseudo-Riemannian structure has the following properties:

\begin{prop}
$$\begin{array}{l}
D_V W= 0\\
D_V X= 0\\
D_X V= [X,V]\\
D_X Y= \mathcal{H}D_X Y + \mathcal{V}D_X Y,\\
\end{array}$$
with $\mathcal{H}D_X Y$ basic $\pi$-related to $\nac_{\cx}\cy$ and
$\mathcal{V}D_X Y$ linear and $(\widehat{\mathcal{V}D_X
Y})(\hv)(\cz)=\hv(\ccr(\cy,\cz)\cx)$.
\end{prop}

\begin{proof}[Proof of the proposition] Use the Koszul formula:
\begin{eqnarray*}2 \langle D_X Y, Z \rangle & = & X\cdot \langle Y, Z \rangle +
Y\cdot \langle Z, X \rangle - Z\cdot \langle X, Y \rangle\\
& & +\langle [X,Y], Z \rangle - \langle [X,Z], Y \rangle - \langle
[Y,Z], X \rangle,
\end{eqnarray*}
for any sections $X, Y, Z$ of $TM$.

Let now $X,Y,Z$ be basic and $U, V, W$ constant. Using
proposition~\ref{basic_results} it is easy to see that $\langle D_V
W, U \rangle=0$ and $\langle D_V W, Z \rangle=0$, so $D_V W=0$.

We have also easily $\langle D_V X, U \rangle=0$. $\langle D_V X,
Y\rangle=0$ follows from the fact that $\nac$ is torsion-free. So
$D_V X=0$.

For $\xi\in T^*B$ we have
\begin{eqnarray*}2\langle D_X Y, U\rangle_\xi & = & (\cx \cdot
\hu(\cy))_{\pi(\xi)} + (\cy \cdot \hu(\cx))_{\pi(\xi)}\\
& & +(\hu([\cx,\cy]))_{\pi(\xi)}-((\nah_{\cx}\hu)(\cy))_{\pi(\xi)}-((\nah_{\cy}\hu)(\cx))_{\pi(\xi)}\\
& = & 2(\hu(\nac_{\cx}\cy))_{\pi(\xi)}
\end{eqnarray*}
From this follows that the horizontal part of $D_X Y$ is basic and
$\pi$-related to $\nac_{\cx}\cy$.

\begin{eqnarray*}2\langle D_X Y, Z\rangle_\xi & = & \langle [X,Y], Z \rangle_\xi - \langle [X,Z], Y \rangle_\xi - \langle
[Y,Z], X \rangle_\xi\\
& = & \langle Id_\xi(\hr(\cx,\cy)\xi), Z \rangle_\xi - \langle
Id_\xi(\hr(\cx,\cz)\xi), Y \rangle_\xi - \langle
Id_\xi(\hr(\cy,\cz)\xi), X
\rangle_\xi\\
& = & -\xi(\ccr(\cx,\cy)\cz)+ \xi(\ccr(\cx,\cz)\cy) +
\xi(\ccr(\cy,\cz)\cx)\\
& = & 2 \xi(\ccr(\cy,\cz)\cx),
\end{eqnarray*}
where the latter equality follows from the first Bianchi identity
for $\nac$. As a consequence the vertical part of $D_X Y$ is linear
and associated to the given section of $End_B(M)$.

\end{proof}

\begin{rem}
The connection defined in this section fits in the scheme described in section~\ref{connection_D}. We have:
$\hphi=\hphi_1$, or explicitly
$\hphi(\cx,\cy)(\hv)(\cz)=\frac{1}{2}((\hr(\cz,\cx)\hv)(\cy)+(\hr(\cz,\cy)\hv)(\cx))$,
\end{rem}

\subsubsection{A symplectic structure on $T^*B$}
We can endow $M=T^*B$ with an almost symplectic structure $\omega$ in the
following way: For $V,W$ vertical sections of $TM$ associated to
$\hv$, $\hw$ and $X,Y$ basic sections of $TM$ $\pi$-related to
$\cx$, $\cy$ fix: $\omega(V, W)=0$,
$\omega(X, Y)=0$, $\omega(V,X)_x=-\omega(X,V)_x= \hv(\cx)_{\pi(x)}$.

It is an easy verification (using basic resp. constant vertical vector fields, the fact that $\nac$ is torsion-free and applying the first Bianchi-identity for $\ccr$) that $d\omega=0$.

\begin{prop}
The given almost symplectic structure is a symplectic structure.
\end{prop}

\begin{rem}
It is an easy verification that the symplectic structure is simply the differential of the Liouville form on $T^*B$ and independent of the choice of the horizontal spaces.
\end{rem}

The following statement makes explicit under which condition a connection $D$ on the manifold $M=T^*B$ of the type discussed in this section preserves the symplectic structure:
\begin{prop}\label{symplectic_phi}
Let $\omega$ be a symplectic structure on the manifold $M=T^*B$ as fixed in this paragraph.
A connection $D$ of the type described in section \ref{connection_D} verifies $D\omega=0$ if and only if
for any vertical constant vector field $V$ and any basic vector fields $X, Y, Z$ we have:
$$\frac{1}{2}(\hr(\cy,\cz)\hv)(\cx)+(\hphi(\cx,\cz)\hv)(\cy)-(\hphi(\cx,\cy)\hv)(\cz)=0.$$
\end{prop}

\begin{proof}
Make the condition $D\omega=0$ explicit using basic and vertical constant vector fields.
\end{proof}

\begin{prop}
$\hphi$ verifies the conditions of proposition \ref{symplectic_phi} if and only if
it is of the form $\hphi_{\frac{1}{3}}+ \hs$ where $(\hs(\cx,\cy)\hu)(\cz)=-\hu(\cs(\cx,\cy)\cz)$
and $\cs$ is in $\Gamma(S^3(T^*B)\otimes TB)$.
\end{prop}

\begin{proof}
Check that $\hphi_{\frac{1}{3}}$ satisfies the condition of proposition \ref{symplectic_phi} and that if $\hphi$ does, $\hs:=\hphi-\hphi_{\frac{1}{3}}$
verifies: $(\hs(\cx,\cy)\hu)\cz=(\hs(\cx,\cz)\hu)\cy$.
\end{proof}

Other interesting structures on $M=T^*B$ are obtained for $\hphi=\hphi_t$ when $t=0$ or $t=-1$.

\subsection{Structures on $TB$}
For $M=TB$ and $t\in\reels$, let $\hphi_{t} :=t\Gamma$. One can consider as before the cases $t=-1$, $t=0$, $t=1$ etc. We will see later that may be the structure corresponding to $t=1$ is most important.

After these examples we return to the general case.

\subsection{Further properties}

\begin{rem}
Note that from $D_V A - D_A V = [V,A]$ follows $D_V A = [V,A]$ for
$A$ linear and $V$ constant, as $D_A V=0$. Similarly $D_X A =
[X,A]$ for $X$ basic and $A$ linear vertical.
\end{rem}

\begin{prop}
For $A$ and $B$ linear vertical vector fields, $D_A B$ is linear
vertical associated to $\hb \circ \ha$.
\end{prop}

The curvature tensor $R$ of the connection $D$ can be characterized
as follows:
\begin{prop}\label{curvature}
For $U, V$ vertical constant and $X, Y, Z$ basic we have:
\begin{enumerate}[(i)]
\item $R(U,V)=0$
\item $R(X,U)V=0$, $R(X,U)Y$ is vertical constant associated to
$\frac{1}{2}\hr(\cx,\cy)\hu+\hphi(\cx,\cy)\hu$.
\item $R(X,Y)U$ is vertical constant associated to
$\hr(\cx,\cy)\hu$.

\noindent $\mathcal{H}(R(X,Y)Z)$ is basic associated to
$\ccr(\cx,\cy)\cz$.

\noindent $\mathcal{V}(R(X,Y)Z)$ is linear and associated to
$\frac{1}{2}(\nah_{\cz}\hr)(\cx,\cy)
-(\nah_{\cx}\hphi)(\cy,\cz)+(\nah_{\cy}\hphi)(\cx,\cz).$
\end{enumerate}
\end{prop}

\begin{corr} In the case $M=T^*B$ equipped with the pseudo-riemannian structure of section~\ref{cotangent} for $U, V$ vertical constant and $X, Y, Z$ basic we have:
\begin{enumerate}
\item[(ii)] $\widehat{R(X,U)Y}(\cz)=\hu(\ccr(\cy,\cz)\cx)$.

\item[(iii)] if $A=\mathcal{V}(R(X,Y)Z)$,
$\ha(\hv)(\ct)=((\nah_{\ct} \hr)(\cx,\cy)\hv)(\cz)$.
\end{enumerate}
\end{corr}

Note $\hpsi:=\frac{1}{2}\hr + \hphi$

\begin{lem}
For $U,V$ vertical constant, $A$ linear and $X, Y, Z$ basic we have:
\begin{enumerate}[(i)]

\item $R(A,U)V=0$, $R(A,U)Z=0$;
\item $R(U,V)A=0$;
\item $R(X,A)U=0$, $R(X,A)Z$ is vertical linear associated to $\hpsi(\cx,\cz)
\circ \ha$;
\item $R(X,U)A=0$;
\item $R(X,Y)A$ is vertical linear associated to $\hr(\cx,\cy)\circ
\ha$.
\end{enumerate}
\end{lem}

\begin{prop}
For $U, V, W$ vertical constant and $T, X, Y, Z$ basic we have:
\begin{enumerate}[(i)]
\item $(D R)(W,U,V)=0$;
\item $(D R)(X,U,V)=0$;
\item $(D R)(W,X,U)=0$;
\item $(D R)(X,Y,U)(V)=0$, $(D R)(X,Y,U)(Z)$ is vertical constant associated to $(\nah
\hpsi)(\cx,\cy,\cz)\hu$;
\item $(D R)(W,X,Y)(V)=0$; $(D R)(W,X,Y)(Z)$ is vertical constant
associated to
$-(\nah\hpsi)(\cx,\cy,\cz)\hw+(\nah\hpsi)(\cy,\cx,\cz)\hw$;
\item $(D R)(T,X,Y)(V)$ is vertical constant associated to
$(\nah \hr)(\ct,\cx,\cy)\hv$. $\mathcal{H}(D R)(T,X,Y)(Z)$ is basic
$\pi$-related to $(\nac_{\ct}\ccr)(\cx,\cy)\cz$. $\mathcal{V}(D
R)(T,X,Y)(Z)$ is linear associated to $-(\nah \nah
\hpsi)(\ct,\cx,\cy,\cz)+(\nah \nah
\hpsi)(\ct,\cy,\cx,\cz)+\hpsi(\ct,\ccr(\cx,\cy)\cz)+\hpsi(\cy,\cz)\circ
\hpsi(\ct,\cx)-\hpsi(\cx,\cz)\circ \hpsi(\ct,\cy)-\hr(\cx,\cy)\circ
\hpsi(\ct,\cz)$.
\end{enumerate}
\end{prop}

\subsection{Symmetric spaces}\label{ss}
\begin{prop}
In the case $M=T^*B$ equipped with the pseudo-riemannian structure of section~\ref{cotangent} ({\em i.e.} for $\hphi=\hphi_{1}$), $\nac \ccr=0$ implies $D R=0$. Further if $B$ is a symmetric space, then $T^*B$ with this structure is a symmetric space as well.
\end{prop}

\begin{prop}
In the case $M=TB$ equipped with the connection corresponding to $\hphi=\hphi_{1}$, $\nac \ccr=0$ implies $D R=0$. Further if $B$ is a symmetric space, then $TB$ with this structure is a symmetric space as well.
\end{prop}

\subsection{Parallel transport}\label{pt}
Suppose $c$ is a (smooth) vertical curve in $M$, {\em i.e.} $\dot c(t)$ is
vertical for any $t$. A vector field along $c$ is parallel if and
only if it can locally be written $t \mapsto (Y+V)_{c(t)}$ with $Y$
basic and $V$ vertical constant.

Suppose now $c$ is any non vertical (smooth) curve in $M$. Locally $c$ is
an integral curve of some vector field $X+U$. Locally any vector field
along $c$ can be extended and written $t \mapsto (Y+V)_{c(t)}$ with $Y$ basic and
$V$ vertical constant. It is parallel if $(D_{\dot c(t)}
(Y+V))_{c(t)}=0$, which can be written $(D_{X+U} (Y+V))_{c(t)}=0$. By splitting this
equation up into horizontal and vertical components we obtain:
$(\nac_{\cx} \cy)_{(\pi \circ c)(t)}=0$ and $(\nah_{\cx} \hv)_{(\pi
\circ c)(t)} + \hpsi(\cx,\cy)_{(\pi \circ c)(t)}(c(t))=0$.

The parallel transport $\tau_c$ along some closed curve $c$, such
that $c(0)=O_{\pi(c(0))}$, can be decomposed as follows: Write $c'$
for the curve $t \mapsto O_{\pi(c(t))}$.
The preceding equations show
that parallel transport of a vertical vector $v\in \mathcal{V}_{x}$ (in $x=c(0)$) along $c$ coincides
with parallel transport of the vector $v\in \mathcal{V}_{O_{\pi(x)}}$ (with the identification $\mathcal{V}_{x}\simeq \mathcal{V}_{O_{\pi(x)}}$) along $c'$. From the
equations one can also deduce that the horizontal part of parallel
transport of a horizontal vector $v$ along $c$ is exactly parallel
transport of $v$ along $c'$.

Note $\tilde c$ the path obtained by concatenating $c$ and
$(c')^{-1}$. $\tau_c$ is the product $\tau_{c'} \circ \tau_{\tilde
c}$. As in the classical proof of the Ambrose-Singer theorem one can
show that $\tau_{\tilde c}$ is generated by the $(\tau_\gamma^*
R)(X,U)$ where $X$ is any horizontal vector in the origin $c(0)$, $U$
any vertical vector in the origin and $\gamma$ any path from the
origin to any point $p$ between $c'(t)$ and $c(t)$ for any $t$.

Note $\cg:=\pi \circ \gamma$. By proposition~\ref{curvature},
$(\tau_\gamma^* R)(X,U)(Y)$ is vertical associated to
$$\tau^{-1}_{\cg}
\hpsi_{\pi(p)}(\tau_{\cg}(\cx),\tau_{\cg}(\cy))\tau_{\cg}(\hu).$$

\subsection{More examples}
Let $B$ be a two dimensional manifold equipped with a torsion-free
non flat connection $\nac$. $T^*B$ equipped with the canonical
pseudo-Riemannian metric of section~\ref{cotangent} and its
Levi-Civita connection $D$ admits then the following holonomy
algebra:
$$Hol(D)=\left\{ \left( \begin{smallmatrix} -^{t}a & b\\0 & a\end{smallmatrix}
\right)|\; a\in Hol(\nac), b=-^{t}b \right\}.$$

To see this, apply the discussion of section~\ref{pt}. The condition
$b=-^{t}b$ follows from the fact that $M$ is pseudo-Riemannian. It
is clear that in dimension $2$, $b$ is necessarily a multiple of
$\left(
\begin{smallmatrix}0 & 1\\-1 & 0\end{smallmatrix} \right)$. It is
enough to show that at least in one point of $M$, $R(X,U)$ is non
vanishing. But this follows from the non flatness of $\nac$.

Note that any connected subgroup of $Gl(2,\reels)$ is a restricted
holonomy group of a torsion-free connection $\nac$ on a two
dimensional manifold(see \cite{krantz}).

So we obtain for every candidate of indecomposable metric holonomy of signature (2,2) with a single totally isotropic invariant plane
listed in \cite{LBBAI} a corresponding metric.

Note
$J=\left(\begin{smallmatrix} 0 & 1\\-1 & 0\end{smallmatrix} \right)$;


$K_{2,\lambda} := \left\{ \left( \begin{smallmatrix} \alpha & \gamma\\0 & \alpha^\lambda \end{smallmatrix} \right); \; \alpha, \gamma\in\reels, \alpha>0 \right\}$ with $\lambda\in\reels$;

$K_{2,\infty} := \left\{ \left( \begin{smallmatrix} 1 & \gamma\\0 & \alpha \end{smallmatrix} \right); \; \alpha, \gamma\in\reels, \alpha>0 \right\}$;

$K_3 := \left\{ \left( \begin{smallmatrix} \alpha & \gamma\\0 & \beta \end{smallmatrix} \right); \; \alpha, \beta, \gamma\in\reels, \; \alpha>0, \beta>0 \right\}$;

$K_4 := SL(2,\reels)$;

$K_5 := GL^+(2,\reels)$;

$K_6 := \left\{ \left( \begin{smallmatrix} \alpha & \beta\\-\beta & \alpha \end{smallmatrix} \right); \; \alpha, \beta\in\reels,\; \alpha^2+\beta^2>0 \right\}$;

$U_1$ the connected group associated to the Lie algebra
$\left\{\left(\begin{smallmatrix} \alpha A & 0\\0 & -\alpha^{t}\!A \end{smallmatrix} \right); \alpha\in\reels \right\}$ with $A=\left(\begin{smallmatrix} 1 & 1\\0 & 1 \end{smallmatrix} \right)$;

$U_{2,\lambda}$ the connected group associated to the Lie algebra
$\left\{\left(\begin{smallmatrix} \alpha C & 0\\0 & -\alpha^{t}\!C \end{smallmatrix} \right); \alpha\in\reels \right\}$, with $C=\left(\begin{smallmatrix} \lambda & 1\\-1 & \lambda \end{smallmatrix} \right)$ and $\lambda\in\reels$, $\lambda\ge 0$;

$\mathbb{A}$ the connected group associated to the Lie algebra $\left\{\left(\begin{smallmatrix} 0 & \alpha J\\0 & 0 \end{smallmatrix} \right); \alpha\in\reels\right\}$;

$\mathbb{B}$ the connected group associated to the Lie algebra $\left\{\left(\begin{smallmatrix} \alpha N & \beta J\\0 & -\alpha^{t}N \end{smallmatrix} \right); \alpha, \beta\in\reels\right\}$ with $N=\left(\begin{smallmatrix} 0 & 1\\0 & 0 \end{smallmatrix}\right);$

\begin{theo}
The restricted holonomy groups of an indecomposable non irreducible semi-riemannian manifold of signature $(2,2)$ leaving invariant a single totally isotropic plane are up to an isomorphism $\mathbb B$, $U_1\cdot \mathbb{A}$, $U_{2,\lambda}\cdot \mathbb{A} (\lambda\in\reels)$, $K_{2,\lambda}\cdot\mathbb{A} (\lambda\in\reels\cup\{\infty\})$,
$K_6\cdot \mathbb{A}$, $K_3\cdot \mathbb{A}$, $K_4\cdot \mathbb{A}$ and $K_5\cdot \mathbb{A}$.

\end{theo}

For $(B,\nac)$ a locally symmetric space of dimension $2$, $Hol_o(\nac)$ is either
$SO_0(2)$, $SO_0(1,1)$ or $\left\{\left(\begin{smallmatrix} 1 & t\\0
& 1\end{smallmatrix}\right) | \; t\in\reels \right\}$, and then
$T^*B$ equipped with the preceding pseudo-riemannian metric is a locally symmetric space.

\end{document}